\documentclass[11pt,a4paper,reqno]{amsart}
\usepackage{enumitem, amsmath, amsthm, amssymb, bbm, amsbsy, textgreek, mathrsfs}

\makeatletter

\usepackage{datetime}
\usepackage[pdfborder={0 0 0}]{hyperref}
  \hypersetup{
    citebordercolor=.1 .6 1,
    colorlinks = false
}

\usepackage{upref,cases}
\hypersetup{citebordercolor=.1 .6 1}

\usepackage[T1]{fontenc}
\usepackage[utf8]{inputenc}

\usepackage[final]{graphicx}

\usepackage{geometry}
\let\epsilon\varepsilon
\newcommand{\eps}{\epsilon}

\renewcommand{\Im}{\ensuremath{\operatorname{Im}}}

\DeclareMathOperator{\Tr}{Tr}

\renewcommand{\kappa}{\varkappa}

\numberwithin{equation}{section}
\newtheoremstyle{ttheorem}%
       {1.8ex\@plus1ex}                
       {2.1ex\@plus1ex\@minus.5ex}      
       {\itshape}           
       {0pt}                   
       {\bfseries}          
       {.}                  
       {.5em}               
       {}                

\newtheoremstyle{ddefinition}%
       {1.8ex\@plus1ex}                
       {2.1ex\@plus1ex\@minus.5ex}      
       {}           
       {0pt}                   
       {\bfseries}           
       {.}                  
       {.5em}               
       {}                

\newtheoremstyle{rremark}%
       {1.8ex\@plus1ex}                
       {2.1ex\@plus1ex\@minus.5ex}      
       {\normalfont}        
       {0pt}                   
       {\bfseries}           
       {.}                  
       {.5em}               
       {}                   

\theoremstyle{ttheorem}
\newtheorem{theorem}{Theorem}[section]
\newtheorem{lemma}[theorem]{Lemma}
\newtheorem{proposition}[theorem]{Proposition}
\newtheorem{corollary}[theorem]{Corollary}

\theoremstyle{ddefinition}

\theoremstyle{rremark}
\newtheorem{remark}[theorem]{Remark}
\newtheorem{myremarks}[theorem]{Remarks}
\newtheorem{myexamples}[theorem]{Examples}
\newtheorem{example}[theorem]{Example}

\newenvironment{remarks}{\begin{myremarks}\begin{nummer}}%
    {\end{nummer}\end{myremarks}}
    {\end{nummer}\end{myexamples}}

\newcounter{numcount}
\newcommand{\labelnummer}{(\roman{numcount})}%

\makeatletter
\providecommand{\showkeyslabelformat}[1]{\relax}        
\let\mysaveformat\showkeyslabelformat                   %
\def\myformat#1{\raisebox{-1.5ex}{\mysaveformat{#1}}}   %

\newenvironment{nummer}%
  {\let\curlabelspeicher\@currentlabel%
    \begin{list}{\textup{\labelnummer}}%
      {\usecounter{numcount}\leftmargin0pt%
        \topsep0.5ex\partopsep2ex\parsep0pt\itemsep0ex\@plus1\p@%
        \labelwidth2.5em\itemindent3.5em\labelsep1em%
      }%
    \let\saveitem\item%
    \def\item{\saveitem%
      \def\@currentlabel{\curlabelspeicher\kern.1em\labelnummer}}%
    \let\savelabel\label%
    \def\label##1{{\ifnum\thenumcount=1\let\showkeyslabelformat\myformat\fi\savelabel{##1}}%
										{\def\@currentlabel{\labelnummer}%
									 	\let\showkeyslabelformat\@gobble
									 	\savelabel{##1item}%
										}%
	   							}%
  }{\end{list}}%

  {\let\curlabelspeicher\@currentlabel%
    \begin{list}{\textup{\labelnummer}}%
      {\usecounter{numcount}\leftmargin0pt%
        \topsep0.5ex\partopsep2ex\parsep0pt\itemsep0ex\@plus1\p@%
        \labelwidth2.5em\itemindent0em\labelsep1em%
        \leftmargin2.5em}%
    \let\saveitem\item%
    \def\item{\saveitem%
      \def\@currentlabel{\curlabelspeicher\kern.1em\labelnummer}}%
    \let\savelabel\label%
    \def\label##1{{\ifnum\thenumcount=1\let\showkeyslabelformat\myformat\fi\savelabel{##1}}%
										{\def\@currentlabel{\labelnummer}%
									 	\let\showkeyslabelformat\@gobble
									 	\savelabel{##1item}%
										}%
    							}%
  }{\end{list}}%

\usepackage{enumitem}

\let\OldItem\item
\newcommand{\MyItem}[2][]{}%
\def\section{\@startsection{section}{1}%
  \z@{1.3\linespacing\@plus\linespacing}{.5\linespacing}%
  {\normalfont\bfseries\centering}}

\def\subsection{\@startsection{subsection}{2}%
  \z@{.8\linespacing\@plus.5\linespacing}{-1em}%
  {\normalfont\bfseries}}

\def\nlsubsection{\@startsection{subsection}{2}%
  \z@{.8\linespacing\@plus.5\linespacing}{.1ex}%
  {\normalfont\bfseries}}

\let\@afterindenttrue\@afterindentfalse%

\renewenvironment{proof}[1][\proofname]{\par \normalfont
  \topsep6\p@\@plus6\p@ \trivlist 
  \item[\hskip\labelsep\scshape
    #1\@addpunct{.}]\ignorespaces
}{%
  \qed\endtrivlist
}

\def\ps@firstpage{\ps@plain
  \def\@oddfoot{\normalfont\scriptsize \hfil\thepage\hfil
     \global\topskip\normaltopskip}%
  \let\@evenfoot\@oddfoot
  \def\@oddhead{
    \begin{minipage}{\textwidth}
      \normalfont\scriptsize
      \emph{\insertfirsthead}
    \end{minipage}}
  \let\@evenhead\@oddhead 
}

\def\insertfirsthead{}

\def\@cite#1#2{{%
 \m@th\upshape\mdseries[{#1}{\if@tempswa, #2\fi}]}}
\newcommand{\C}{\mathbb{C}}
\newcommand{\N}{\mathbb{N}}

\newcommand{\R}{\mathbb{R}}
\newcommand{\Z}{\mathbb{Z}}
\renewcommand{\le}{\leqslant}

\renewcommand{\leq}{\leqslant}
\renewcommand{\geq}{\geqslant}

\DeclareMathOperator{\tr}{Tr}

\def\sumd{\displaystyle\sum}

\let\<\langle
\let\>\rangle

\newcommand{\1}{1}

\newcommand{\upd}{\mathrm{d}}
\renewcommand{\d}{\upd}   

\newcommand{\hairspace}{\kern .04167em}

\DeclareMathOperator{\TextIm}{Im}

\renewcommand{\Im}{\TextIm}

\DeclareMathOperator{\sech}{sech}

\newcommand{\beq}{\begin{equation}}
\newcommand{\eeq}{\end{equation}}

\def\sumd{\displaystyle\sum}

\def\intd{\displaystyle\int}

\def\clap#1{\hbox to 0pt{\hss#1\hss}}

\makeatother
\begin{document}

\begin{abstract}
In this note, we study the asymptotics of the determinant $\det(I_N - \beta H_N)$ for $N$ large, where $H_N$ is the $N\times N$ restriction of a Hankel matrix $H$ with finitely many jump discontinuities in its symbol satisfying $\|H\|\leq 1$. Moreover, we assume $\beta\in\C$ with $|\beta|<1$ and $I_N$ denotes the identity matrix. We determine the first order asymtoptics as $N\to\infty$ of such determinants and show that they exhibit power-like asymptotic behaviour, with exponent depending on the height of the jumps. For example, for the $N \times N$ truncation of the Hilbert matrix $\mathbf{H}$ with matrix elements $\pi^{-1}(j+k+1)^{-1}$, where $j,k\in \Z_+$ we obtain
$$
\log \det(I_N - \beta \mathbf{H}_N) = -\frac{\log N}{2\pi^2}
			\big(\pi\arcsin(\beta)+\arcsin^2(\beta)+o(1)\big),\qquad N\to\infty.
$$
\end{abstract}

\sloppy

\title[On determinants identity minus Hankel matrix]{On determinants identity minus Hankel matrix}

\author[E.\ Fedele]{Emilio Fedele}
\address[E.\ Fedele]{Department of Mathematics, King's College London, Strand, London, WC2R 2LS, UK}
\email{emilio.fedele@kcl.ac.uk}

\author[M.\ Gebert]{Martin Gebert}
\address[M.\ Gebert]{School of Mathematical Sciences, Queen Mary University of London, Mile End
Road, London E1 4NS, UK}
\curraddr{ Department of Mathematics, University of California, Davis, Davis, CA 95616, USA}

\email{mgebert@math.ucdavis.edu}
\thanks{M.G. was supported in part by the European Research Council starting grant SPECTRUM (639305)}


\maketitle

\section{Introduction and results}
Given a bounded function $f$ on the unit circle $\mathbb{T}:=\{v \in \C\, :\, |v|=1\}$, the associated Hankel matrix $H(f):\ell^2(\Z_+)\to\ell^2(\Z_+),\, \Z_{+}:=\big\{0,\,1,\,2,\,\ldots\big\}$, is given by its matrix elements
\beq
(H(f))(n,m) = \widehat{f}(n+m),\qquad n,m\in\Z_+,
\eeq
where for $k\in \Z_+$
\beq\label{hankel}
\widehat{f}(k):= \frac 1 {2\pi} \int_0^{2 \pi}  f(e^{it }) e^{-i k t}\d t.
\eeq
The function $f$ is called the \textit{symbol} of the matrix $H(f)$. 

Throughout this note, we restrict our attention to symbols which have finitely-many jump discontinuities and satisfy $\sup_{z\in\mathbb T}|f(z)|\leq 1$.  We will make more precise assumptions later on, see conditions (A)\,-\,(C). Hankel matrices with jump discontinuities in their symbols are well-studied \cite{Wi:66,Pow-Disc, Pow-Hankel,MR1949210} and still attract attention in the operator theory community, see e.g. \cite{MR3214675}.

Our goal here is to study the large $N$ behaviour of
$
 \det(I_N - \beta H_N(f)),
$
where $H_N(f)$ is the $N\times N$ restriction of the Hankel matrix $H(f),\, \beta \in \C$ so that $|\beta|<1$ and $I_N$ is the identity matrix. Our assumption on the boundedness of the symbol ensures that $\|H(f)\|\leq 1$ and so $\|\beta H(f)\|\leq |\beta |<1$.
We will compute the first order term in its asymptotic expansion for large $N$ and show that
\beq\label{eqn:decay}
\det(I_N - \beta H_N(f))=N^{-\gamma_f(\beta)+ o(1)}
\eeq
 as $N\to\infty$,
where the  exponent $\gamma_f(\beta)\in\C$ depends explicitly on the location of the jumps as well as their height, see our main result Theorem \ref{thm:general} for the precise formulation. 

To illustrate our result, consider the following two explicit Hankel matrices 
\begin{gather}\label{eqn:Hilbmat}
\mathbf{H}:=H(\psi)=\Big\{\frac 1 {\pi(n+m+1)}\Big\}_{n,m\in\Z_{+}} \quad\text{and}\quad
\mathbf{S}:=H(\eta)=\Big\{\frac{\sin(\pi(n+m)/2)}{\pi(n+m)}\Big\}_{n,m\in\Z_{+}},
\end{gather}
with the convention that $\mathbf{S}(0,0)=1/2$, where the symbols $\psi$ and $\eta$ are given in Example \ref{example} later on. The matrix $\mathbf{H}$ is the \textit{Hilbert matrix} and is well-known in the literature. 
For these Hankel matrices Theorem \ref{thm:general} states that the asymptotic formula \eqref{eqn:decay} holds with
\beq \label{eqn:hilbexp}
\gamma_\psi(\beta)=\frac{1}{2\pi^2}\left(\pi \arcsin(\beta)+\arcsin^{2}(\beta)\right)\quad\text{and}\quad
\gamma_\eta(\beta)=\frac{1}{\pi^{2}}\arcsin^{2}\Big(\frac{\beta}{2}\Big).
\eeq
The different expressions for $\gamma_\psi(\beta)$ and $\gamma_\eta(\beta)$ in the two cases are related to their symbols having jumps located differently on $\mathbb T$. In the case of $\mathbf H$, the symbol has only one jump located at 1, causing the appearance of both the linear and the quadratic $\arcsin$ term in \eqref{eqn:hilbexp}. In the case of $\mathbf S$, however, the symbol has jumps at the conjugate points $\pm i$ and the linear arcsin term does not appear. In general, only jumps at $\pm 1\in\mathbb T$ will cause a linear $\arcsin$ term whereas an $\arcsin^2$ term will always occur if there are jumps at conjugate points on $\mathbb T$. We also note that $\gamma_\psi(\beta)<0$ for $\beta\in(-1,0)$ and therefore we have in this case power-like growth in \eqref{eqn:decay}. 

The problem which we study here fits into the more general framework of asymptotics of determinants of Hankel, Toeplitz and Hankel plus Toeplitz matrices. These are well-studied objects, see for example \cite{MR1679634,MR1845906,MR2831118,MR3078693, MR3644515} and references therein. Exhaustive answers to various questions related to the asymptotics of Toeplitz and Hankel determinants have been found, however the behaviour of completely general Hankel plus Toeplitz determinants is not entirely understood yet. In most known results, the Hankel and Toeplitz matrix are related to the same symbol. We prove here a first order asymptotic formula for a simple class of Hankel plus Toeplitz determinants which, to the best of our knowledge, does not fall directly in the cases considered before.

We end the introduction with a word about the proof. The first step in studying our problem is to make use of the series expansion of the logarithm $\log (1-z) = -\sum_{n\in\N} z^{n}/n$ valid for $|z|<1$ which implies
\beq
\log\, \det\left(I_N- \beta H_N(f)\right)=\Tr\,\log\left(I_N- \beta H_N(f)\right) = \sum_{n\in\N} \beta^n\Tr H_N(f)^n/ n.
\eeq
The fact that the series expansion is only valid for $|z|<1$ is the reason why we take $|\beta|<1$. The asymptotic behaviour of $\Tr H_N(f)^n$ is found in Lemma \ref{lm:asymptPowers}, and it partially follows from \cite[Theorem 5.1]{Wi:66}, where this was obtained for the simpler case of the Hilbert matrix with only one jump in its symbol. Surprisingly, the first order contributions in the asymptotic expansion of $\tr H_N(f)^n$ are the coefficients of the power series of $\arcsin$ and $\arcsin^2$ times $\log N$, see Proposition \ref{p:log_integral}.

\subsection{Model and results}
As we saw earlier on, the Hankel matrix $H(f):\ell^2(\Z_+)\to\ell^2(\Z_+)$ is determined by its matrix elements
\beq
(H(f))(n,m) = \widehat f (n+m),\qquad n,m\in\Z_+,
\eeq
where $\widehat f(k)$ is defined in \eqref{hankel} for $k\in \Z_+$. It is clear that $H(f)$ depends linearly on $f$. Throughout the paper, we make the following assumptions on the symbol $f$:
\begin{itemize}
\item[(A)] for all $z\in\mathbb T$ the following limits exist
\beq
f(z^+):=\lim_{\eps \searrow 0} f(e^{i \eps}z)\qquad\text{and}\qquad f(z^-):= \lim_{\eps \searrow 0} f(e^{-i \eps}z).
\eeq 
The points where the limits do not coincide are called \textit{jump-discontinuities} and we only assume a finite number of them.
\item[(B)] with $\Omega$ denoting the set of all discontinuities of $f$, we assume that $f \in C^{\gamma}(\mathbb{T}\backslash \Omega)$, for some $1/2<\gamma\leq 1$ and some $C>0$ that for all $\delta>0$ and all $z\in \Omega$
\beq\label{Holder}
|f(z^+) -  f(e^{i \delta}z)|\leq C\delta^\gamma\quad \text{and}\quad  |f(z^-) -  f(e^{-i \delta}z)|\leq C\delta^\gamma;
\eeq
\item[(C)] it satisfies the bound
\beq\label{bound:symbol}
\sup_{v\in\mathbb T} |f(v)| \leq 1.
\eeq
\label{ABC}
\end{itemize}
We write $f \in PD(\mathbb{T})$ if it satisfies all of the above assumptions. For future reference we define for $z\in\mathbb T$
\beq\label{defkappa}
\kappa_z(f) := \frac{f(z^+)- f(z^-)} 2
\eeq
and refer to this as the \textit{height of the jump.} 
\begin{remarks}
\item Assumption (B) is only of technical nature. It simplifies the proofs and for most relevant examples these assumptions are satisfied.
\item The bound \eqref{bound:symbol} guarantees $|\kappa_z(f)|\leq 1$ for all $z\in\mathbb T$ and that the operator $H(f)$ satisfies $\|H(f)\|\leq 1$, see \cite{MR1949210}. 
\end{remarks}
\begin{example}\label{example}
The most important example of a symbol fitting in our framework is given by 
\begin{gather}\label{eqn:HilbSymbol}
\psi(e^{i t}) =i\pi^{-1} e^{-it} (\pi - t),\ \ \  t\in [0,2\pi).
\end{gather}
Integration by parts shows that this is a symbol for the Hilbert matrix given in \eqref{eqn:Hilbmat}, i.e. that $\mathbf{H}=H(\psi)$.
Another example of a function in this class is given by 
\beq
\eta(e^{i t})= 1_{\{\cos t>0\}},\ \ \ t\in [0,2\pi)
\eeq
with jump-discontinuities at the points $\pm i$. This is a symbol for the Hankel matrix $\mathbf S=H(\eta)$ given in \eqref{eqn:Hilbmat}.
\end{example}

As before, let $H_N(f)$ denote the $N\times N$ restriction of the infinite matrix $H(f)$, i.e. let $H_{N}(f):=1_{N}H(f)1_{N}$, where $1_{N}$ is the orthogonal projection onto the span of $\{e_{j}\}_{j=0}^{N-1}$, where $e_{j},\, j\in\Z_+$, are the standard basis vectors of $\ell^2(\Z_+)$. Setting
\begin{equation}\label{def:HN}
	D^{\beta}_N(f) := \det(I_N - \beta H_N(f)),
\end{equation}
our main result is
\begin{theorem}\label{thm:general}
Suppose $f\in PD(\mathbb T)$. Let $\Omega \subset \mathbb{T}$ be the set of its jump discontinuities. For $\beta\in\C$ with $|\beta|<1$ we have
\begin{align}\label{thm:eq2} 
\log D^{\beta}_N (f) =  -\dfrac{\gamma_f(\beta)}{2\pi^2} \log N +o(\log N),
\end{align}
as $N\to\infty$, where 
\beq \label{thm:eq2.1}
\gamma_f(\beta):=\pi\big(\arcsin(-i\beta \kappa_1(f))+\arcsin(-i\beta \kappa_{-1}(f))\big)+ \sum_{z \in \Omega}\arcsin^2(-i \beta  (\kappa_{z}(f)\kappa_{\overline z}(f))^{1/2}).
\eeq
\end{theorem}

\begin{remarks}
\item  The expression in \eqref{thm:eq2.1} is independent of the choice of the analytic branch of the square root. This follows from the power series 
\beq
\arcsin^2(v) =  \dfrac{1}{2}\sumd\limits_{m=1}^{\infty} 
		\dfrac{\left(m!\right)^2 4^m v^{2m}}{\left(2m\right)! m^2},\quad |v|\leq 1,
\eeq	
which implies that $\arcsin^2$ is a function of $v^2$. 
\item It is evident that we have a non-zero contribution in \eqref{thm:eq2} only if both $z$ and $\overline z$ are jump-discontinuities of the symbol $f$. For example, in the case of self-adjoint Hankel matrices the jump discontinuities only appear in pairs $z,\overline z$ and $\gamma_f(\beta)\neq 0$ in this case. 
Moreover, the terms $\arcsin\left(-i\beta \kappa_{\pm 1}(f)\right)$ only appear for jumps at $z=\pm 1$. This is yet another manifestation of the subtle differences  between jumps at $z=\pm 1$ compared to those located at $z\in\mathbb T\setminus\{\pm1\}$, see for example \cite{Pow-Hankel,MR3475464}. 
\item Even though the proof of the theorem relies on  $|\beta|<1$, we believe that the above asymptotic formula holds for $|\beta|=1$. Indeed, using different methods, this can be achieved for the special case of the Hilbert matrix $\mathbf{H}$ given in \eqref{eqn:Hilbmat}. In this case one can prove
\beq
\log \det(I_N- \mathbf{H}_N) =  - \frac{\gamma_\psi(1)}{2\pi^2}\log N + o\left(\log N\right)
\eeq
as $N\to\infty$, see \cite{GP}, where
\beq
\frac{\gamma_\psi(1)}{2\pi^2} =  \frac{1}{2\pi^2}\left(\pi \arcsin(1)+\arcsin^{2}(1)\right) =\frac 3 8.
\eeq
The authors use the explicit diagonalization of the Hilbert matrix and their methods cannot immediately be generalized to arbitrary Hankel matrices with jump discontinuities in the symbol considered here.

\end{remarks}

Using our methods, one can also consider asymptotics of determinants related to powers of Hankel matrices. For instance, one can prove the following

\begin{corollary}\label{coro:square}
Let $0\leq \beta< 1$ and, as before, denote by $\mathbf{H}$ the Hilbert matrix.
Then as $N\to\infty$
\beq
 \log\det( I_N - \beta^{2} 1_N\mathbf{H}^{2}\1_N) = -\frac {1}{ \pi^2}\arcsin^2(\beta)\log N   + o(\log N). 
\eeq
\end{corollary}

\begin{remark}
Determinants of the above form appear in the study of the asymptotic behaviour of ground-state overlaps of many-body fermionic systems.
In this context Corollary~\ref{coro:square} gives a partial answers to a question asked in \cite[Rmk.~2.7]{MR3376020}. We will not explain the problem here and refer
to \cite{MR3551180,MR3341967} for a precise formulation and further 
reading about the relation of the problem to determinants of Hankel operators.
\end{remark}

\section{Proof of Theorem \ref{thm:general}}
In the following we denote by $S^p$ the standard Schatten-$p$-class and by $\|\cdot\|_p$ its norm for $p\geq 1$. 

Let $f\in PD(\mathbb T)$ and we write for brevity $\kappa_z=\kappa_z(f)$. 
Since the operator norm of Hankel operators satisfies $\|H(f)\|\leq \sup_{v\in\mathbb T} |f(v)|$, assumption (C) implies for all $N\in\N$ that 
\beq
\|\beta H_N(f)\| \leq |\beta|<1.
\eeq
Hence the series expansion 
 \beq\label{log:series}
 \log(1-v) =-\displaystyle \sum_{k\in \N} \frac{v^k}k
 \eeq
  valid for all $|v|<1$ implies that 
	\begin{align}\label{e:log_det}
	\log\det ( I - \beta H_N(f)) 
		= -\sumd\limits_{k=1}^{\infty}\beta^k \dfrac{\Tr H_N(f)^k}{k}.
	\end{align}

In the next step we compute the asymptotics of $\Tr H_N(f)^k$ when $N\to\infty$. This is the main part of the proof. 

\begin{theorem}\label{thm:tr:asympt}
Let $f\in PD(\mathbb T)$ and $\Omega$ be the set of its jumps discontinuities. We denote by $B$ the Beta function. Then, for $k\in\N$ odd, we obtain
\beq\label{thm:tr:asympteq1}
\Tr H_N(f )^k =  \frac{\kappa_{1}^k+ \kappa_{-1}^k}{2\pi^2} (-i)^k  B\Big(\frac{k}{2},
			\frac{1}{2}\Big) \log N + o(\log N)
\eeq
as $N\to\infty$ and for $k\in\N$ even we obtain 
\begin{align}
\Tr H_N(f )^k
	=  
\sum_{z\in\Omega}	\displaystyle\frac{(\kappa_{z}\kappa_{\overline z})^{k/2}}{2\pi^2} (-i)^k B\Big(\frac{k}{2},
			\frac{1}{2}\Big) \log N + o(\log N)
\label{eq122}	
\end{align}
as $N\to\infty$. 
\end{theorem}

In particular, it follows that for any $k\in\N$ the following limits exists
\beq
\lim_{N\to\infty} \frac{\Tr H_N(f )^k}{\log N} =: \mu_k(f)
\eeq
where $\mu_k(f)\in\C$ is given in \eqref{thm:tr:asympteq1}, respectively \eqref{eq122}. Moreover, we need the following proposition.

\begin{proposition}\label{cor:tr:asympt} 
Let $f\in PD(\mathbb T)$. Then 
\beq
\limsup_{N\to\infty} \frac{\|H_N(f )\|_2^2 }{\log N}<\infty.
\eeq
\end{proposition}

We prove Theorem \ref{thm:tr:asympt} and Proposition \ref{cor:tr:asympt} in Section \ref{sec:pfThm2}. Theorem \ref{thm:tr:asympt} might be of independent interest. We also need one more proposition to prove Theorem \ref{thm:general}.

\begin{proposition}\label{p:log_integral}
	Let $|v|\leq 1$, then the series
	\begin{equation*}
		S(v):= \sumd\limits_{m=1}^{\infty}v^m
		\dfrac{B\left(\frac{m}{2},\frac{1}{2}\right)}{2\pi^2 m}\qquad\text{and}\qquad
		T(v) := \sumd\limits_{m=1}^{\infty}v^{2m}
		\dfrac{B\left(m,\frac{1}{2}\right)}{4\pi^2 m}.
	\end{equation*}
    are absolutely convergent and the following identities hold
	\begin{equation}\label{e:log_integral}
		S(v) = \dfrac{1}{2\pi}\arcsin(v)
		+\dfrac{1}{2\pi^2}\arcsin^2(v)\quad \text{and}\quad
	T(v) = \dfrac{1}{2\pi^2}\arcsin^2(v).
	\eeq
\end{proposition}

\begin{remark}
	The series $S(v)$ can also be written as an integral. A computation shows that
	\beq\label{eq:remark}
		\dfrac{1}{\pi} \intd\limits_{0}^{\infty} 
			\sech^m\left(u \pi\right)\d u = \dfrac{1}{2\pi^2}B\Big(\frac{m}{2},
			\dfrac{1}{2}\Big).
	\eeq
	Hence, Fubini's theorem implies for $|v|\leq 1$ that
	\begin{equation}\label{e:log_integral_taylor}
	-S(v) =  - \frac 1 \pi \sumd\limits_{m=1}^{\infty} \intd\limits_{0}^{\infty}
	\dfrac{\left(v\sech(\pi u)\right)^m}{ m}\d u = 
	\dfrac{1}{\pi}\intd\limits_{0}^{\infty}\log\left(1-v\sech(\pi u)\right)\d u.
	\end{equation}
\end{remark}

\begin{proof}[Proof of Proposition~\ref{p:log_integral}]
	We split the sum $S(v)$, $|v|\leq 1$,  in two parts, one corresponding to the odd terms and the even ones. The odd contribution is
	\begin{align}
		I^{\left(\mbox{odd}\right)}\left(v\right) = \dfrac{1}{2\pi^2}
		\sumd\limits_{m=0}^{\infty}
		\dfrac{B\left(m+\frac{1}{2},\frac{1}{2}\right)v^{2m+1}}{2m+1}
		&= \dfrac{1}{2\pi}\sumd\limits_{m=0}^{\infty}
		\dfrac{\left(2m\right)!v^{2m+1}}{4^m \left(m!\right)^2\left(2m+1\right)}\notag\\
		&= \dfrac{1}{2\pi}\arcsin(v).\label{Iodd}
	\end{align}
	The even contribution to the sum is
	\begin{align}
		I^{\left(\mbox{even}\right)}\left(v\right) = \dfrac{1}{2\pi^2}
		\sumd\limits_{m=1}^{\infty} \dfrac{B\left(m,\frac{1}{2}\right)v^{2m}}{2m}
		&= \dfrac{1}{4\pi^2}\sumd\limits_{m=1}^{\infty} 
		\dfrac{\left(m!\right)^2 4^m v^{2m}}{\left(2m\right)! m^2}\notag\\
		&= \dfrac{1}{2\pi^2} \arcsin^2(v).\label{Ieven}
	\end{align}
	Here, we used the power series expansions for $\arcsin$ and $\arcsin^2$ stated in \cite[(1.641) and (1.645)]{MR2360010} which is absolutely convergent for $|v|\leq 1$. This gives the result where we note that $T(v)$ is the same as $I^{\left(\mbox{even}\right)}\left(v\right)$.
\end{proof}

Given Theorem \ref{thm:tr:asympt}, Proposition \ref{cor:tr:asympt}  and Proposition \ref{p:log_integral}, we are in position to prove Theorem~\ref{thm:general}.

\begin{proof}[Proof of Theorem \ref{thm:general}]
Since by assumption $\|\beta H_N(f)\|<1$, we use the series expansion \eqref{e:log_det} and obtain for any $M\in\N$ that 
\beq\label{lm:pf:thm:eq1}
\log \det(I_N - \beta H_N(f)) = -\sum_{k=1}^M \beta^k \dfrac{\Tr H_N(f)^k}{k} - \sum_{k=M+1}^\infty \beta^k \dfrac{\Tr H_N(f)^k}{k}.
\eeq
First, we focus on
\beq
\limsup_{N\to\infty}\Big|\frac 1 {\log N} \sum_{k=M+1}^\infty \beta^k \dfrac{\Tr H_N(f)^k}{k}\Big|.
\eeq
To do so, we use $ \|H_N( f)\|\leq 1$ to obtain the inequality 
\beq
|\Tr H_N( f)^k| \leq  \|H_N(f)\|^{k-2} \|H_N( f)\|_2^2\leq \| H_N(f)\|_2^2
\eeq
 valid for $k\in \N$ and $k\geq 2$ which yields for $M>1$ that
\begin{align}
\Big|\sum_{k=M+1}^\infty \beta^m \dfrac{\Tr H_N(f)^k}{k}\Big| \leq \sum_{k=M+1}^\infty \beta^k \dfrac{\| H_N(f)\|_2^2}{k}
 \leq  \| H_N(f)\|_2^2 \frac {\beta^{M+2}}{1-\beta}. 
\end{align}
Proposition \ref{cor:tr:asympt}  implies that 
$
\displaystyle\limsup_{N\to\infty} \|  H_N(f)\|_2^2\big/\log N = \mu
$
for some $\mu\in\R$
and therefore since $|\beta|<1$ we have that 
\beq
\limsup_{M\to\infty}\limsup_{N\to\infty} \Big|\frac 1 {\log N} \sum_{k=M+1}^\infty \beta^k \dfrac{\Tr H_N(f)^k}{k}\Big| 
\leq 
 \mu\, \limsup_{M\to\infty}  \frac {\beta^{M+2}}{1-\beta}
= 0. 
\eeq
Plugging this into \eqref{lm:pf:thm:eq1} and recalling that $\displaystyle\lim_{N\to\infty} \Tr  H_N(f)^k/\log N= \mu_k(f)$, $k\in\N$, by Theorem \ref{thm:tr:asympt}, we obtain that 
\begin{align}
\limsup_{N\to\infty} \frac{\log \det(I_N - \beta H_N(f))}{\log N} 
\leq&
\limsup_{M\to\infty}\,\Big(-\sum_{k=1}^M \frac{\beta^k} k \mu_k(f) \Big)\notag\\
&+
\limsup_{M\to\infty}\limsup_{N\to\infty} \Big|\frac 1 {\log N} \sum_{k=M+1}^\infty \beta^k \dfrac{\Tr H_N(f)^k }{k}\Big|\notag \\
=&
\limsup_{M\to\infty}\,\Big(-\sum_{k=1}^M \frac{\beta^k} k \mu_k(f)\Big).
\end{align}
Since $| \kappa_z \kappa_{\overline z}|\leq 1$ for all $z\in \Omega$ by assumption, the sum 
 $\sum_{k=1}^\infty \frac{\beta^k} k \mu_k(f)$ for $|\beta|\leq 1$ is absolutely convergent, see Proposition \ref{p:log_integral}. This implies that 
 \beq
 \limsup_{M\to\infty}\Big(-\sum_{k=1}^M \frac{\beta^k} k \mu_k(f)\Big)=\, -\sum_{k=1}^\infty \frac{\beta^k} k \mu_k(f).
 \eeq
Along the very same lines we also obtain that
\beq
\liminf_{N\to\infty} \frac{\log \det(I_N - \beta H_N(f))}{\log N} \geq-\sum_{k=1}^\infty \frac{\beta^k} k \mu_k(f)
\eeq
and therefore we end up with
\begin{align}
\lim_{N\to\infty} \frac{\log \det(I_N - \beta H_N(f))}{\log N} 
=&
  -\sum_{k=1}^\infty \frac{\beta^{k}} {k}\mu_k(f)\label{abc}
\end{align}
and the power series in Proposition \ref{p:log_integral} below give the result. 
\end{proof}

\begin{proof}[Proof of Corollary \ref{coro:square}]
To prove Corollary \ref{coro:square} one uses the expansion \eqref{e:log_det} and obtains for $0\leq \beta<1$
\beq
\log \det( I_N -  \beta^{2}1_N\mathbf H^21_N) =- \sum_{k\in\N} \beta^{2k} \frac {\Tr (1_N\mathbf H^{2}1_N)^k}{k}.
\eeq
For the rest of the proof we use the abbreviations $A:= 1_N\mathbf H^{2}1_N$ and $B:= \mathbf H_N^2$. Then 
\beq\label{core:eq1}
\Tr A^k  - \Tr B^k  = \sum_{j=0}^{k-1} \Tr  A ^j(A-B)  B^{k-1-j}
\eeq
and H\"older's inequality implies 
\begin{align}
\big| \Tr  A ^j(A-B)  B^{k-1-j} \big|   
\leq \|A\|^j \|A-B\|_1 \|B\|^{k-1-j}.
\end{align}
From the definition of $\mathbf H$ we obtain $\|A\|,\|B\|\leq 1$. Moreover, by the positivity of $A-B$, we obtain
\beq\label{core:eq2}
\|A-B\|_1 = \Tr 1_N \mathbf H (1-1_N) \mathbf H 1_N= \|1_N \mathbf H (1-1_N)\|_2^2 = O(1)
\eeq
as $N\to\infty$, where the last inequality follows easily from the explicit matrix elements of $\mathbf H$. Equations \eqref{core:eq1}--\eqref{core:eq2} imply 
\beq
\Tr (1_N\mathbf H^{2}1_N)^n  = \Tr \mathbf H_N^{2n} + O(1).
\eeq
For the latter we computed the first order asymptotics as $N\to\infty$ in Lemma \ref{lm:asymptPowers}. Then the assertion follows from Propostition \ref{p:log_integral} along the very same lines as Theorem \ref{thm:general}. 
\end{proof}

\section{Analysis of the model operator}

To prove Theorem \ref{thm:tr:asympt}, we first investigate a family of model operators related to the Hilbert matrix introduced in \cite[Chap. 10.1]{MR1949210}. We recall the Hilbert matrix $\mathbf H:=H(\psi)$ introduced in \eqref{eqn:Hilbmat} with symbol
\beq
\psi(e^{it})= \frac 1 \pi i e^{-i t} (\pi - t),\quad t\in[0,2\pi).
\eeq
In particular,  one has $\|\mathbf H\|=1$. We define the following model symbols for $z\in\mathbb T$
\begin{align}\label{def:psiz}
\psi_{z}(e^{it})&:=\frac 1 {i}\psi(\overline z e^{it}),\ \ \ t\in [0,\, 2\pi).
\end{align}
For any $z \in \mathbb T$ this function satisfies
\beq\label{jumpskappa}
\psi_{z}(z^+) - \psi_{z}(z^-) = 2. 
\eeq
Furthermore, the corresponding Hankel matrix $H(\psi_z)$ admits the representation
\beq\label{defHankelrep}
H(\psi_z)= \frac{1}{ i} U_z \mathbf H U_z 
\eeq
where $U_z:\ell^2(\Z_+)\to\ell^2(\Z_+)$ is the unitary operator  given by  $(U_z x)(n):= z^n x(n)$ for $x\in \ell^2(\Z_+)$, $z\in\mathbb T$ and $n\in\Z_+$. In particular, one can compute the matrix elements explicitly and one obtains for $z\in\mathbb T$
\beq\label{kernel}
(H(\psi_z))(n,m) = \frac{1}{ i} \frac{z^{n+m}}{n+m+1}. 
\eeq
%

The large $N$ asymptotics of traces of powers of the model operators can be computed explicitly:

\begin{lemma}\label{lm:asymptPowers}
We denote by $B$ the Beta-function. Let $k\in\N$. Then, for $a\in\C$ we obtain that
\begin{align}\displaystyle
\Tr H_N(a \psi_{z})^k =  \begin{cases} \displaystyle\frac{(-i)^k (a)^k}{2\pi^2}  B\Big(\frac{k}{2},
			\frac{1}{2}\Big) \log N + o(\log N), & z\in\pm 1\\
	o(\log N), &
	z\in\mathbb T\setminus\{\pm 1\}\label{eq1}
	\end{cases}
\end{align}
as $N\to\infty$ while for $z\in\mathbb T\setminus\{\pm 1\}$ and  $a,b\in\C$ we obtain that
\begin{align}
 \Tr H_N(a \psi_z + b \psi_{\overline z})^k
	=\begin{cases}
	O(1),&\ k\in\N\ \text{odd},\\
	\displaystyle\frac{(-i)^k(ab)^{k/2}}{\pi^2} B\Big(\frac{k}{2},
			\frac{1}{2}\Big) \log N + o(\log N),&\ k\in\N\ \text{even}
	\end{cases}
\label{eq12}	
\end{align}
as $N\to\infty$.
\end{lemma}
We prove Lemma \ref{lm:asymptPowers} in Section \ref{pf:lmtr}.

\section{Proof of Theorem \ref{thm:tr:asympt}}\label{sec:pfThm2}

Let $f\in PD(\mathbb T)$ and let $\Omega$ be the set of its jump discontinuities and we define
\beq
\Psi:=  \sum_{z\in\Omega} \varkappa_z\psi_z.
\eeq
The definition of $\psi_z$ in \eqref{def:psiz}, the identity \eqref{jumpskappa} and the definition of $\kappa_z=\kappa_z(f)$ in \eqref{defkappa}, imply that the jumps of $f$ and $\Psi$ are located at the same points and the heights of the jumps are the same, i.e. $\kappa_z(f) = \kappa_z(\Psi)$ for all $z\in \mathbb T$. Moreover, by assumption (B),  $f\in \C^\gamma(\mathbb T\setminus \Omega)$ for some $1/2<\gamma\leq 1$ and clearly $\Psi\in C^\infty(\mathbb T\setminus\Omega)$ which implies
 \beq\label{11111}
 f - \Psi \in C^\gamma(\mathbb T).
\eeq
We first prove Proposition~\ref{cor:tr:asympt}.

\begin{proof}[Proof of Proposition \ref{cor:tr:asympt}]
We first note that the definition of Besov spaces $B_2^{1/2}(\mathbb T)$ given in \cite[eq. (A.2.10)]{MR1949210} implies that 
$C^\gamma(\mathbb T) \subset B_2^{1/2}(\mathbb T)$ for $1/2<\gamma\leq 1$. Hence, \cite[Chap. 6, Thm.~2.1]{MR1949210}
and \eqref{11111} imply that
\beq\label{pclass}
H(f) - H(\Psi) \in S^{2}
\eeq
for all $\gamma>1/2$. Hence, using the above and Jensen's inequality we obtain
\beq\label{1}
\|H_N(f)\|^2_2 \leq 2 \|H(\Psi)\|^2_2 + O(1)
\eeq
as $N\to\infty$. The explicit representation of the matrix entries of $H(\Psi)$ in \eqref{kernel} implies
\beq\label{22}
\|H(\Psi)\|_2^2\leq C \sum_{n=0}^{N-1} \sum_{m=0}^{N-1} \frac 1{(n+m+1)^2},
\eeq
for some constant $C>0$. Estimating the latter double sum by the corresponding integral, we obtain
\beq\label{3}
\sum_{n=0}^{N-1} \sum_{m=0}^{N-1} \frac 1{(n+m+1)^2} = O(\log N)
\eeq
as $N\to\infty$. This, together with \eqref{1} and \eqref{22}, gives the assertion. 
\end{proof}

\begin{lemma}\label{lem:phips3}
Let $k\in\N$. Then the asymptotic formula
\beq
 \Tr H_N(f )^k   =  \Tr  H_N(\Psi)^k + o(\log N)
\eeq
 holds as $N\to\infty$.
\end{lemma}

\begin{proof}
As before, by \eqref{11111} we obtain
$
H(f) - H(\Psi) \in S^{2}
$
for all $1/2<\gamma\leq 1$, see \eqref{pclass}. Moreover, \eqref{11111} also  implies that the Fourier coefficients of $f - \Psi$ are absolutely summable, see \cite[Thm. 1.13]{MR3052498}.

\noindent We write  $H( f) = H(\Psi) + A$, where $A:= H( f- \Psi)$ and set $A_N:=1_N A1_N$. For $k=1$, using the  absolute summability of the Fourier coefficients of $f - \Psi$, we obtain that 
\beq
\big|\Tr H_N( f)  - \Tr  H_N(\Psi)\big| \leq \sum_{i=0}^{N-1} \big|(H_N( f- \Psi))(i,i)\big| = O(1),
\eeq
as $N\to\infty$. 
For $k\geq 2$, the identity
\beq\label{lemeq2}
\Tr  H_N(f) ^k  = \Tr H_N(\Psi)^k  + \sum_{j=0}^{k-1} \Tr  H_N(f) ^j A_N H_N(\Psi)^{k-1-j} 
\eeq
holds. To control the error, we use $A\in S^2$. The cyclicity of the trace and the H\"older inequality for $S^p$ classes, implies for $1\leq j \leq k-1$ 
\begin{align}
\big|\Tr  H_N(f)  ^j A_N H_N(\Psi)^{k-1-j} \big| 
&\leq \|A_N\|_2 \|H_N(\Psi)^{k-1-j}H_N(f) ^j\|_2\notag\\
&\leq C^{k-1-j}\|A\|_2 \|H_N(f) \|_2,\label{1234}
\end{align}
where we used that $\|H_N(\Psi)\|\le C$ for some constant $C>0$ independent of $N$, $\|H_N(f)\|\leq 1$ and the standard inequality $\|CD\|_2\leq \|C\| \|D\|_2$ valid for compact operators $C$ and $D$. Proposition~\ref{cor:tr:asympt} implies that 
\beq
 \|H_N(f) \|_2 = O(\log N)^{1/2} = o(\log N)
\eeq
as $N\to\infty$. For $j=0$ we use in \eqref{1234} the bound 
\beq
\big|\Tr  A_N H_N(\Psi)^{k-1-j} \big| \leq C^{k-2-j} \|A\|_2 \|H_N(\Psi) \|_2.
\eeq
The asymptotic formula
$
 \|H_N(\psi) \|_2 = O(\log N)^{1/2} = o(\log N)
$
holds as well for $N\to\infty$, see \eqref{2} and \eqref{3}. 
This together with \eqref{1234} gives the result. 
\end{proof}

In view of Lemma \ref{lm:asymptPowers} we divide the set of discontinuities in
\beq
\Omega = \Omega_1\cup\Omega_2,
\eeq
where
\beq
 \Omega_1:=\big\{\pm 1\big\} \cup  \big\{z\in \Omega:\ \overline z\notin \Omega\big\}\quad\text{and}\quad \Omega_2:=\big\{z\in\Omega:\ \Im z>0,\  \overline z\in\Omega\big\}.
\eeq
With this notation at hand we show
\begin{lemma}\label{lem:phipsi}
Let $k\in\N$. Then 
\beq
 \Tr H_N(\Psi )^k   =  \sum_{z\in\Omega_1}  \Tr H_N(\varkappa_z\psi_z)^k+\sum_{z\in\Omega_2}  \Tr H_N(\varkappa_z\psi_z +\varkappa_{\overline z} \psi_{\overline z})^k  + O(1)
\eeq
as $N\to\infty$. 
\end{lemma}

\begin{proof}
We first prove that 
\beq\label{lm:eqeq1}
 \Tr H_N(\Psi )^k   =  \Tr 1_N H(\Psi)^k1_N + O(1).
\eeq
From \cite[Thm. 1.2]{LaSa} we infer that for some constant $C_k>0$ depending on k
\beq\label{lapSaf}
\big| \Tr H_N(\Psi )^k-  \Tr 1_N H(\Psi  )^k1_N \big| \leq C_k \|1_N H(\Psi) (1-1_N)\|^2_2.
\eeq
The explicit representation of the kernel of $H(\psi_z)$ in \eqref{kernel} implies for some constant $C>0$ that
\beq\label{2}
\|1_NH(\Psi)(1-1_N)\|_2^2\leq C \sum_{n=0}^{N-1} \sum_{m=N}^{\infty} \frac 1{(n+m+1)^2}<\infty.
\eeq
To prove the assertion we note that for $z,w\in\Omega$ with $z\neq w$ and $z\neq \overline w$
\beq\label{traceclass}
H(\psi_z )H(\psi_w)\in S^1,
\eeq
 which is proven in  \cite[Lem. 2.5]{MR3475464}. This implies
\beq\label{wer}
  \Tr 1_NH(\Psi)^k1_N= \sum_{z\in\Omega_1}  \Tr 1_N H(\varkappa_z\psi_z )^k 1_N  +\sum_{z\in\Omega_2} \Tr 1_N H(\varkappa_z\psi_z +\varkappa_{\overline z}\psi_{\overline z} )^k 1_N +  O(1).
\eeq
The same argument as in \eqref{lapSaf} yields 
\beq
\Tr 1_N H(\psi_z)^k 1_N  = \Tr  H_N(\psi_z)^k  + O(1)
\eeq
and 
\beq
\Tr 1_N H(\varkappa_z\psi_z +\varkappa_{\overline z}\psi_{\overline z} )^k 1_N  = \Tr  H_N(\varkappa_z\psi_z +\varkappa_{\overline z}\psi_{\overline z} )^k  + O(1).
\eeq
This gives the assertion together with \eqref{lm:eqeq1} and \eqref{wer}. 
\end{proof}

\begin{proof}[Proof of Theorem \ref{thm:tr:asympt}]
The theorem follows directly from  Lemma \ref{lem:phips3}, Lemma \ref{lem:phipsi} and the asymptotics deduced in Lemma \ref{lm:asymptPowers}.
\end{proof}

\section{Proof of Lemma \ref{lm:asymptPowers}}\label{pf:lmtr}

\begin{proof}[Proof of Lemma \ref{lm:asymptPowers}]
The statement for $z=1$ follows directly from \cite[proof of Thm. 4.3]{Wi:66}, see especially \cite[eq. (12)]{Wi:66}, where it is proven that
\beq\label{Widom1}
\Tr H_N(\psi_{1})^{k}=\frac{(-i)^{k}}{2\pi^{2}}B\Big(\frac{k}{2}, \frac{1}{2}\Big)\log N + o(\log N),
\eeq
as $N\to \infty$. As before, $B$ denotes the Beta function. We remark that the result in the paper cited above has been corrected to take into account a factor of $1/2\pi$ missing in the computations of \cite[proof of Thm. 4.3]{Wi:66}. Similar results to the above are true in greater generality, see \cite{Fed}.

\noindent As we mentioned earlier on in \eqref{defHankelrep} , we have $$H(\psi_{-1})=U_{-1} \mathbf H U_{-1},$$ where $U_{-1}$ is the unitary and self-adjoint operator of multiplication by the sequence $(-1)^{n}$ on $\ell^{2}(\Z_+)$. Therefore, the result of \cite{Wi:66} gives
\beq\label{Widom2}
\Tr H_N(\psi_{-1})^k = \frac{(-i)^{k}}{2\pi^{2}}B\Big(\frac{k}{2}, \frac{1}{2}\Big)\log N + o(\log N).
\eeq
This and \eqref{Widom1} give the first part of \eqref{eq1}.

Next we consider the case $z\in \Omega\setminus\{\pm 1\}$ and note that the second part of \eqref{eq1} follows from \eqref{eq12} with $b=0$. Therefore, we only prove \eqref{eq12}. Let $a,b\in\C$. First we note that the unitary $U_z$ and the projection $1_N$ commute.
Using representation \eqref{defHankelrep}, we expand $H_N(a\psi_z+b\psi_{\overline z})^k= (-i)^k\big(a U_z \mathbf H_N U_z + b U_{\overline z} \mathbf H_N U_{\overline z}\big)^k$ in $2^k$ terms and obtain
\begin{align}\label{eq11}
H_N(a\psi_z+b\psi_{\overline z})^k
&= (-i)^k\big( aU_z \mathbf H_N U_z + b U_{\overline z} \mathbf H_N U_{\overline z}\big)^k\notag\\
&= \begin{cases} (-i)^k (a b)^{(k-1)/2} a \, U_z \mathbf H_N^k U_z + (a b)^{(k-1)/2} b\,   U_{\overline z} \mathbf H_N^k U_{\overline z}  + A_1,&\ k\ \text{odd}\\
(-i)^k (ab)^{k/2}\big(U_z \mathbf H_N^k U_{\overline z} + U_{\overline z} \mathbf H_N^k U_{z}\big)  + A_2,&\ k\ \text{even}
\end{cases}
\end{align}
for some operators $A_1$ and $A_2$. We first deal with the errors $A_1$ and $A_2$. The operators $A_1$ and $A_2$ consist of a sum of $2^k-2$ terms and each summand has at least one factor $\mathbf H_N U_{ z^2} \mathbf H_N$ or $\mathbf H_N U_{\overline z ^2} \mathbf H_N$. More precisely, any factor of $A_1$ is either of the form
\beq
a^r b^s U_z\mathbf H_N U_z \cdots U_z\mathbf H_N U_{ z^2} \mathbf H_N U_z \cdots \mathbf H_N U_{\overline z} 
\eeq
for some $r,s\in \N$ with $s+r=k$ or the adjoint of the latter. 
Since $\Im\, z\neq 0$, we have from \cite[Chap. 10, Lem. 1.2]{MR1949210} that the matrix elements of $\mathbf H_N U_{z ^2} \mathbf H_N$ satisfy
\beq
\big|\big(\mathbf H_N U_{z^2} \mathbf H_N(n,m)\big)\big| \leq \frac 2 {|1- z^2|} \frac 1{(1+m)(1+n)}.
\eeq
Using this and  the pointwise bound on the matrix elements of $U_z \mathbf H_N U_z$ of the form $\big|(U_z \mathbf H_N U_z)(n,m) \big|\leq\frac {1} {\pi(n+m+1)}$ ,$n,m\in\Z_+$,  we estimate
\begin{align}
\big|\Tr 1_N & U_z\mathbf H U_z \cdots U_z\mathbf H_N U_{ z^2} \mathbf H_N U_z \cdots \mathbf H_N U_{\overline z} 1_N \big|\notag\\
&= \Big|\sum_{n_1,\cdots, n_k=1}^\infty \big(U_z\mathbf H_N U_z\big)(n_1,n_2)\cdots \big(\mathbf H_N U_{z^2} \mathbf H_N\big)(n_p,n_{p+1})\cdots \big(U_{\overline z} \mathbf H_N U_{\overline z}\big)(n_{k},n_1)\Big|\notag\\
&\leq \frac 2 {|1-z|^2} \big|\< x,\mathbf H_N^{k-1} x\>\big|\notag\\
&\leq \frac 2 {|1-z|^2}\|x\|_2^2<\infty\label{boundmatrix}
\end{align}
for some  $p\in \N$, where we defined $x\in\ell^2(\Z_+)$ with $x(n):=1/(n+1)$.
Writing out all terms of $\tr A_i  $, $i=1,2$, explicitly in terms of its matrix elements and using a bound of the form \eqref{boundmatrix} implies for $i=1,2$ that,
as $N\to\infty$, 
\beq
\big|\Tr  A_i  \big| = O(1).
\eeq

For $k\in\N$ odd we obtain 
\begin{align}\label{eq:Leibniz}
&\big|\Tr \big( (a b)^{(k-1)/2}a  U_z \mathbf H_N^k U_z + (ab)^{(k-1)/2}b   U_{\overline z} \mathbf H_N^k U_{\overline z}  \big)\big|\notag\\
=&
|ab|^{(k-1)/2}\big|\Tr \big(a  U_z \mathbf H_N^k U_z + b   U_{\overline z} \mathbf H_N^k U_{\overline z}  \big)^k\big|
\notag\\
\leq&  | a b|^{(k-1)/2}\Big(|a|\big|\sum_{n=0}^{N-1}  z^{2n} \mathbf H_N^k(n,n)   \big|+|b| \big| \sum_{n=0}^{N-1} {\overline z}^{2n}\mathbf H_N^k(n,n)   \big|\Big).
\end{align}
From the explicit matrix elements $\mathbf H_N(n,m)=\frac 1 {\pi(n+m+1)}$ we obtain for all $n\in\Z_+$ that $0\leq \mathbf H_N^k(n+1,n+1)\leq \mathbf H_N^k(n,n)$, i.e. the sequence $a_n:= \mathbf H_N^k(n,n)$, $n\in\Z_+$,  is strictly monotonously decreasing. Now Lemma \ref{Leibniz} below gives, as $N\to\infty$,
\beq
\eqref{eq:Leibniz} =O(1).
\eeq

In the case $k\in \N$ even the definition of $U_z$ yields
\beq
(-i)^k( a b)^{k/2} \Tr \big(U_z \mathbf H_N^k U_{\overline z} + U_{\overline z} \mathbf H_N^k U_{z}  \big)= 2(-i)^k ( ab)^{k/2} \Tr \mathbf H_N^k 
\eeq
but this is just the asymptotics of the Hilbert matrix which was discussed in the first part of the proof. This gives the assertion. 
\end{proof}

\begin{lemma}\label{Leibniz}
Let $z\in\mathbb T\setminus\{1\}$ and $(a_n)_{n\in\Z_+}$ be such that $0\leq a_{n+1}\leq a_n$ for all $n\in\Z_+$. Then
\beq
\big|\sum_{n=0}^{N} z^{n} a_n \big|\leq a_0 \frac 2 {|1- z|}
\eeq
and, in particular,  $\sum_{n=0}^{N} z^n a_n  =O(1)$ as $N\to\infty$. 
\end{lemma}

\begin{proof}
The lemma follows directly from Abel's summation formula
\beq
\sum_{n=0}^{N} z^{n} a_n = B_N a_N + \sum_{k=1}^{N-1} B_k (a_k -a_{k-1})
\eeq
where $B_k=\sum_{l=0}^k z^l$. 
\end{proof}

\section*{Acknowledgements}
We are grateful to Sasha Pushnitski for teaching us about Hankel matrices and suggestions on an earlier version of the paper. We also thank Mihail Poplavskyi for several discussions on the topic.

\end{document}